\documentclass[12pt, reqno]{amsart}
\usepackage{eurosym}
\usepackage{amsmath, amsthm, amscd, amsfonts, amssymb, graphicx, color}
\usepackage[bookmarksnumbered, colorlinks, plainpages]{hyperref}
\usepackage{tipa}
\usepackage{tipx}
\usepackage{float}
\hypersetup{colorlinks=true,linkcolor=red, anchorcolor=green, citecolor=cyan, urlcolor=red, filecolor=magenta, pdftoolbar=true}
\newtheorem{theorem}{Theorem}[section]
\newtheorem{lemma}[theorem]{Lemma}

\theoremstyle{definition}

\newtheorem{example}[theorem]{Example}

\theoremstyle{remark}
\newtheorem{remark}[theorem]{Remark}
\numberwithin{equation}{section}

\begin{document}
	
	\setcounter{page}{1}
	
	
	\begin{center}
		{\Large \textbf{\ On $\text{\texthtq}$-statistical
				approximation of wavelets aided Kantorovich \textit{\text{\texthtq}}-Baskakov operators}} \bigskip
		
		\href{https://orcid.org/0000-0002-2566-3498}{\textbf{Mohammad Ayman-Mursaleen%
		}}$^{1,2}$, \href{https://orcid.org/0000-0002-9184-8941}{\textbf{Bishnu P.
				Lamichhane }}$^{1}$, \href{https://orcid.org/0000-0002-1217-963X}{%
			\textbf{Adem Kiliçman}}$^{2,\ast}$ and 
		\href{https://orcid.org/0000-0001-8614-8281}{\textbf{%
				Norazak Senu}}$^{2,3}$\\[0pt]
		
		\bigskip $^{1}$School of Information and Physical Sciences, \\[0pt]
		\textbf{The University of Newcastle}, University Drive, Callaghan, New South
		Wales 2308, \textbf{Australia}\\[0pt]
		$^{2}$Department of Mathematics and Statistics, Faculty of Science, \\[0pt]
		\textbf{Universiti Putra Malaysia}, 43400 UPM Serdang, Selangor, \textbf{%
			Malaysia}\\[0pt]
		$^{3}$Institute for Mathematical Research (INSPEM), \\[0pt]
		\textbf{Universiti Putra Malaysia}, 43400 UPM Serdang, Selangor, \textbf{%
			Malaysia}\\[0pt]
		
		\bigskip
		
		\href{mailto:mohdaymanm@gmail.com}{mohdaymanm@gmail.com}, \href{mailto:mohammad.mursaleen@uon.edu.au}%
		{mohammad.mursaleen@uon.edu.au}, \href{mailto:mursaleen.ayman@student.upm.edu.my}%
		{mursaleen.ayman@student.upm.edu.my}\\[0pt]
		\href{mailto:bishnu.lamichhane@newcastle.edu.au}{%
			bishnu.lamichhane@newcastle.edu.au}\\[0pt]
		
		\href{mailto:akilicman@yahoo.com}{akilicman@yahoo.com}, \href{mailto:akilic@upm.edu.my}%
		{akilic@upm.edu.my}\\[0pt]
		\href{mailto:norazak@upm.edu.my}{norazak@upm.edu.my}\\[0pt]
		
		$^{*}$Corresponding author \bigskip \bigskip
		
		\textbf{Abstract}
	\end{center}
	
	\parindent=8mm {\footnotesize {The aim of this research is to examine various statistical approximation properties with respect to Kantorovich \textit{\text{\texthtq}}-Baskakov operators using wavelets. We discuss and investigate a weighted statistical approximation employing a Bohman-Korovkin type theorem as well as a statistical rate of convergence applying a weighted modulus of smoothness $\omega_{\rho_{\alpha}}$ correlated with the space  $B_{\rho\alpha}(\mathbb{R_{+}})$ and Lipschitz type maximal functions. Both topics are covered in the article.}}\newline
	

	
	\section{\textbf{\ Introduction and preliminaries}}
	In 1995, Agratini \cite{agratini} introduced a class of Sz\'{a}sz-type operators by means of compactly supported wavelets of Daubechies.
	Later on in 1997, Gonska and Zhou \cite{gon} used the Daubechies'
	compactly-supported wavelets to establish a new class of Baskakov-type
	operators. This technique of employing wavelets in
	modifying the classical operators is very useful which provides a tool to
	achieve the local information of approximation by such operators. In \cite%
	{nasir}, Nasiruzzaman \textit{et al} further modified
	the operators of Gonska and Zhou \cite{gon} by defining their $q$-analog to
	get a better rate of convergence. In this scholarly article, our focus is to delve deeper into the various approximation properties exhibited by the operators described in \cite{nasir}. Our proposed study aims to further enhance our understanding of these operators and their potential applications.

	\parindent8mm Note that that the Bernstein polynomials  \cite{bns} converge uniformly to the value $\text{\textg}(x)$ for every continuous function $\text{\textg}$, where $x$ is any real value between 0 and 1. The following defines the Bernstein polynomials:
	\begin{equation}
		\left( \mathcal{B}_{r,s }^{\ast } \text{\textg}\right) (x)=\sum_{s =0}^{r}\binom{r}{s }%
		x^{s}(1-x)^{r-s }\text{\textg}\left( \frac{s }{r}\right) ,  \label{bnst}
	\end{equation}%
	where $\binom{r}{i}$ refers to the binomial coefficients.

	The Sz\'{a}sz \cite{sbbl4} as well as Baskakov \cite{bsk} operators were formed in approximating the continuous functions which were defined with respect to the unbounded interval $[0,\infty ).$ Here, the Baskakov operators are written as
	
	\begin{equation*}
		\left( \mathcal{B}_{\text{\textlonglegr},s } \text{\textg}\right)
		(x)=\sum_{s =0}^{\infty }\binom{\text{\textlonglegr}+s -1}{ s }\frac{x^{s}}{%
			(1+x)^{\text{\textlonglegr}+s }}\text{\textg}\left( \frac{s }{\text{%
				\textlonglegr}}\right) .
	\end{equation*}%
	Bernstein operators were modified by Kantorovich \cite{kant} and were christened Bernstein-Kantorovich operators. These operators are utilized in approximating the functions of broader classes as opposed to continuous functions. Moreover, the following are the operators that define Bernstein-Kantorovich operators given:
	
	\begin{equation}  \label{brk}
		\left( \mathcal{K}_{\text{\textlonglegr},s }\text{\textg}\right) (x)=(\text{%
			\textlonglegr}+1)\sum_{s =0}^{\text{\textlonglegr}}\binom{\text{\textlonglegr%
		}}{s } x^{s}(1-x)^{\text{\textlonglegr}-s }\int_{\frac{s }{\text{%
					\textlonglegr}+1 }}^{\frac{s +1}{\text{\textlonglegr}+1 }}\text{\textg}(%
		\text{\textrtailt})\mathrm{d}\text{\textrtailt},
	\end{equation}%
	for functions $\text{\textg}\in L_{p}[0,1]$ ($1\leq p<\infty $).
	
	To determine the $L_{p}$-approximation, Ditzian and Totik \cite{d} expressed Kantorovich modification with respect to Baskakov operators, which is called the Baskakov-Kantorovich operators written as
	\begin{equation}
		\left( \mathcal{K}_{m,\text{\textrtaill} }~\text{\textg}\right) (x)=m\sum_{%
			\text{\textrtaill} =0}^{\infty }\binom{m+\text{\textrtaill} -1}{\text{%
				\textrtaill} }\frac{x^{\text{\textrtaill} }}{(1+x)^{m+\text{\textrtaill} }}%
		~\int_{\frac{\text{\textrtaill} }{m}}^{\frac{\text{\textrtaill} +1}{m}}\text{%
			\textg}\left( \text{\textrtailt}\right) \mathrm{d}\text{\textrtailt}.
		\label{k1}
	\end{equation}
	
	The $\text{\texthtq}$-calculus application appeared as a relatively new research field in the approximation theory. Here, the first $\text{\texthtq}$-analogue of the famous Bernstein polynomials was established by Lupa\c{s} \cite{sbbl2} by employing the concept of $\text{\texthtq}$-integers. On the other hand, in 1997, Phillips \cite{philip} took into consideration a different $\text{\texthtq}$-analogue with respect to the classical Bernstein polynomials. Subsequently, numerous researchers investigated the $\text{\texthtq}$-generalizations with regard to a variety of operators by examining their approximation properties. For instance, the $\text{\texthtq}$-variant with respect to Baskakov operators \cite{qbsk} may be written as
	\begin{equation*} \label{1.1}
		\left( \mathcal{V}_{m,\text{\textrtaill} ,\text{\texthtq}}~\text{\textg}%
		\right) (x)=\sum_{\text{\textrtaill} =0}^{\infty }B_{m,\text{\textrtaill} ,%
			\text{\texthtq}}(x)~\text{\textg}\left( \frac{[\text{\textrtaill} ]_{\text{%
					\texthtq}}}{\text{\texthtq}^{\text{\textrtaill} -1}[m]_{\text{\texthtq}}}%
		\right) ,
	\end{equation*}%
	where
	\begin{equation*}
		B_{m,\text{\textrtaill} ,\text{\texthtq}}(x)=\left[
		\begin{array}{c}
			m+\text{\textrtaill} -1 \\
			\text{\textrtaill}%
		\end{array}%
		\right] _{\text{\texthtq}}\frac{x^{\text{\textrtaill} }}{(1+x)_{\text{%
					\texthtq}}^{m+\text{\textrtaill} }}\text{\texthtq}^{\frac{\text{\textrtaill}
				(\text{\textrtaill} -1)}{2}},
	\end{equation*}%
	while the $\text{\texthtq}$-Baskakov-Kantorovich operator \cite{gupta} are
	defined by
	
	\begin{equation}  \label{operator-1}
		\left(\mathcal{T}_{m,\text{\textrtaill} ,\text{\texthtq}} ~ \text{\textg}%
		\right)(x)= [m]_\text{\texthtq} \sum_{\text{\textrtaill} =0}^{\infty} \text{%
			\texthtq}^{\text{\textrtaill} -1}B_{m,\text{\textrtaill} ,\text{\texthtq}%
		}(x) \int_{\frac{\text{\texthtq}[\text{\textrtaill} ]_\text{\texthtq}}{[m]_%
				\text{\texthtq}}}^{\frac{[\text{\textrtaill} +1]_\text{\texthtq}}{[m]_\text{%
					\texthtq}}}\text{\textg}\left( \text{\texthtq}^{1-\text{\textrtaill} } \text{%
			\textrtailt}\right)\mathrm{d}_\text{\texthtq}\text{\textrtailt}.
	\end{equation}
	
	\begin{lemma}
		\label{Lemma 2.1} With respect to the test functions given by $e_{j}=\text{\textrtailt}^{j},~j=0,1,2$, it follows that
		\begin{eqnarray*}
			(1)\quad \left( \mathcal{V}_{m,\text{\textrtaill} ,\text{\texthtq}%
			}~e_{0}\right) (x) &=&1, \\
			(2)\quad \left( \mathcal{V}_{m,\text{\textrtaill} ,\text{\texthtq}%
			}~e_{1}\right) (x) &=&x, \\
			(3)\quad \left( \mathcal{V}_{m,\text{\textrtaill} ,\text{\texthtq}%
			}~e_{2}\right) (x) &=&x^{2}+\frac{x}{[m]_{\text{\texthtq}}}\left( 1+\frac{x}{%
				\text{\texthtq}}\right) .
		\end{eqnarray*}
	\end{lemma}
	
	\subsection{Basics of \text{\texthtq}-Calculus}
	
	The \text{\texthtq}-integer $[m]_{\text{\texthtq}}$, the $\text{\texthtq}$-factorial $%
	[m]_{\text{\texthtq}}!$ as well as the \text{\texthtq}-binomial coefficient are expressed as below (see
	\cite{bbl6}) :\newline
	\begin{align*}
		\lbrack m]_{\text{\texthtq}}& :=\left\{
		\begin{array}{ll}
			\frac{1-\text{\texthtq}^{m}}{1-\text{\texthtq}}, & \hbox{if~}\text{\texthtq}\in \mathbb{R}%
			^{+}\setminus \{1\} \\
			m, & \hbox{if~}\text{\texthtq}=1,%
		\end{array}%
		\right. \mbox{for $m\in \mathbb{N} $~and~$[0]_\text{\texthtq}=0$}, \\
		\lbrack m]_{\text{\texthtq}}!& :=\left\{
		\begin{array}{ll}
			\lbrack m]_{\text{\texthtq}}[m-1]_{\text{\texthtq}}\cdots \lbrack 1]_{\text{%
					\texthtq}}, & \hbox{$m\geq 1$,} \\
			1, & \hbox{$m=0$,}%
		\end{array}%
		\right. \\
		\left[
		\begin{array}{c}
			m \\
			\text{\textrtaill}%
		\end{array}%
		\right] _{\text{\texthtq}}& :=\frac{[m]_{\text{\texthtq}}!}{[\text{%
				\textrtaill} ]_{\text{\texthtq}}![m-\text{\textrtaill} ]_{\text{\texthtq}}!},
	\end{align*}%
	accordingly. Here, the \text{\texthtq}-analogue  with respect to $(1+x)^{m}$ is given by the polynomial
	\begin{equation*}
		(1+x)_{\text{\texthtq}}^{m}:=\left\{
		\begin{array}{ll}
			(1+x)(1+\text{\texthtq}x)\cdots (1+\text{\texthtq}^{m-1}x) & \quad
			m=1,2,3,\cdots \\
			1 & \quad n=0.%
		\end{array}%
		\right.
	\end{equation*}%
	%
	%
	%
	%
	%
	
	\noindent The Gauss binomial formula may be written as
	\begin{equation*}
		(x+a)_{\text{\texthtq}}^{m}=\sum\limits_{\text{\textrtaill} =0}^{m}\left[
		\begin{array}{c}
			m \\
			\text{\textrtaill}%
		\end{array}%
		\right] _{\text{\texthtq}}\text{\texthtq}^{\text{\textrtaill} (\text{%
				\textrtaill} -1)/2}a^{\text{\textrtaill} }x^{m-\text{\textrtaill} }.
	\end{equation*}
	
	On the other hand, the $\text{\texthtq}$-derivative $D_\text{\texthtq} \text{\textg}$ of a
	function $\text{\textg}$ may be written as
	
	\begin{equation*}
		(D_\text{\texthtq} \text{\textg})(x)= \frac{\text{\textg}(x)-\text{\textg}(%
			\text{\texthtq}x)}{(1-\text{\texthtq})x},~ x\neq 0,
	\end{equation*}%
	as well as $(D_\text{\texthtq} \text{\textg})(0)=\text{\textg}^{\prime }(0),$
	provided that $\text{\textg}^{\prime }(0)$ exists. If $\text{\textg}$ is
	differentiable, then
	
	\begin{equation*}
		\lim_{\text{\texthtq} \to 1}D_\text{\texthtq} \text{\textg}(x)= \lim_{\text{%
				\texthtq} \to 1}\frac{\text{\textg}(x)-\text{\textg}(\text{\texthtq}x)}{(1-%
			\text{\texthtq})x}=\frac{d \text{\textg}(x)}{dx}.
	\end{equation*}%
	For $m \geq 1$,
	\begin{equation*}
		D_\text{\texthtq} (1+x)_\text{\texthtq}^m= [m]_\text{\texthtq} (1+\text{%
			\texthtq}x)_\text{\texthtq}^{m-1},~ D_\text{\texthtq} \left(\frac{1}{ (1+x)_%
			\text{\texthtq}^m}\right)= -\frac{[m]_\text{\texthtq}}{(1+x)_\text{\texthtq}%
			^{m+1}},
	\end{equation*}
	
	\begin{equation*}
		D_{\text{\texthtq}}\bigg{(}\frac{u(x)}{v(x)}\bigg{)}=\frac{v(\text{\texthtq}%
			x)D_{\text{\texthtq}}u(x)-u(\text{\texthtq}x)D_{\text{\texthtq}}v(x)}{v(x)v(%
			\text{\texthtq}x)}.
	\end{equation*}%
	The $\text{\texthtq}$-Jackson definite integral is expressed by
	\begin{equation*}
		\int_{0}^{\infty /A}f(x)d_{\text{\texthtq}}x=(1-\text{\texthtq}%
		)\sum_{n=-\infty }^{\infty }f\left( \frac{\text{\texthtq}^{n}}{A}\right)
		\frac{\text{\texthtq}^{n}}{A}\qquad (A\in \mathbb{R}-\{0\}).
	\end{equation*}
	
	\subsection{$\text{\texthtq}$-Statistical convergence}
	
	The definition of $\text{\texthtq}$-analog with respect to Ces\`{a}ro matrix $C_{1}$ is not unique (see \cite%
	{akrh}, \cite{akbe}). Here, we may take into consideration the \text{\texthtq}-Ces\`{a}ro matrix, $%
	C_{1}(\text{\texthtq})=(c_{nk}^{1}(\text{\texthtq}^{k}))_{n,k=0}^{\infty }$ expressed by
	\begin{equation*}
		c_{nk}^{1}(\text{\texthtq}^{k})=%
		\begin{cases}
			\frac{\text{\texthtq}^{k}}{[n+1]_{\text{\texthtq}}}\text{ \ if }k\leq n, \\
			0\text{ \ \ \ \ \ \ \ otherwise}.%
		\end{cases}%
	\end{equation*}%
	which is regular with respect to $\text{\texthtq}\geq 1$.
	
	Suppose $\mathcal{K}\subseteq \mathbb{N}$ (the set of natural numbers). Therefore, $%
	\delta (\mathcal{K})=\lim_{r}\frac{1}{r}\#\{k\leq r~:~k\in \mathcal{K}\}$ is
	known as the asymptotic density with respect to $\mathcal{K}$, in which $\#$ resembles the
	cardinality of the enclosed set. Moreover, a sequence $\eta =(\eta _{k})$ is known as
	statistically convergent to the number $\mathfrak{s}$ provided that $\delta (\mathcal{K}%
	_{\varepsilon })=0$\ for every $\varepsilon >0$, in which $\mathcal{K}%
	_{\varepsilon }=\{k\leq r:|\eta _{k}-\mathfrak{s}|>\varepsilon \}$ (refer to \cite%
	{fas}).
	
	In the current years, Aktu\u{g}lu and Bekar \cite{akbe} determined \text{\texthtq}-density as well as \text{\texthtq}%
	-statistical convergence. The \text{\texthtq}-density may be expressed as
	\begin{equation*}
		\delta _{\text{\texthtq}}(\mathcal{K})=\delta _{C_{1}^{\text{\texthtq}}}(\mathcal{K})=\lim
		\inf_{n\rightarrow \infty }(C_{1}^{\text{\texthtq}}\chi _{\mathcal{K}})_{n}=\lim
		\inf_{n\rightarrow \infty }\sum_{k\in K}\frac{\text{\texthtq}^{k-1}}{[n]},\text{ }\text{\texthtq}\geq 1.
	\end{equation*}
	
	A sequence $\eta =(\eta _{k})$ is known to be \text{\texthtq}-statistically~convergent with respect to
	the number $\mathcal{L}$ provided that $\delta _{\text{\texthtq}}(\mathcal{K}_{\varepsilon })=0$,
	in which $\mathcal{K}_{\varepsilon }=\{k\leq n:|\eta _{k}-\mathcal{L}|\geq
	\varepsilon \}$ for every $\varepsilon >0$. In other words, for each $%
	\varepsilon >0,$.
	\begin{equation*}
		\lim_{n}\frac{1}{[n]}\#\{k\leq n:\text{\texthtq}^{k-1}|\eta _{k}-\mathcal{L}|\geq
		\varepsilon \}=0
	\end{equation*}%
	We may also write $St_{\text{\texthtq}}-\lim \eta _{k}=\mathcal{L}$.
	
	Provided that $\delta (\mathcal{K})=0$ for an infinite set $\mathcal{K}$, $\delta
	_{\text{\texthtq}}(\mathcal{K})=0,$. Therefore, statistical convergence \cite[Example 15]{fas}
	Implies that it is \text{\texthtq}-statistical~convergence but not conversely (refer to [Example 15%
	]\cite{akbe}).
	
	\section{\textbf{Wavelets aided $\text{\texthtq}$-Baskakov-Kantorovich operators}}
	
	We now recall several basic knowledge with respect to wavelets \cite{w1,w2}. Here, the wavelets notion denotes the set of functions of the form given by
	\begin{equation*}
		\Psi _{\mu ,\nu }(x)=\mu ^{-\frac{1}{2}}\Psi \left( \frac{x-\nu }{\mu }%
		\right) ~\mu >0,~\nu \in \mathbb{R},
	\end{equation*}%
	which are formed via translations and dilations with respect to a single function $\Psi $, which is called the mother wavelet or basic wavelet. Moreover, following the Franklin-Stro\"{m}berg theory, the constant $\mu $ may be substituted by $2^{i}$ while $\nu $ may be substituted by $%
	2^{i}\text{\textrtaill} $ having $i$ and $\text{\textrtaill} $ to be the integers. With respect to an arbitrary function $\text{\textg}\in L_{2}(\mathbb{R}),$ the wavelets have a crucial part in the orthonormal basis, in which the $g$ function is given as:
	\begin{equation*}
		\text{\textg}(x)=\sum_{-\infty }^{\infty }\sum_{-\infty }^{\infty }\gamma (i,%
		\text{\textrtaill} )\Psi _{i,\text{\textrtaill} }(x),
	\end{equation*}%
	in which
	\begin{equation*}
		\gamma (i,\text{\textrtaill} )=2^{\frac{i}{2}}\Psi _{i,\text{\textrtaill}
		}(x)\int_{\mathbb{R}}f(x)\Psi (2^{i}x-\text{\textrtaill} )\mathrm{d}x.
	\end{equation*}%
	Daubechies \cite{dau} formed an orthonormal basis for $L_{2}(\mathbb{R})$ expressed by
	\begin{equation*}
		2^{\frac{i}{2}}\Psi _{s}(x)(2^{i}x-\text{\textrtaill} ),
	\end{equation*}%
	where $s$ refers to the non-negative integer, $i,\text{\textrtaill} $ denote the integers set as well as the support of $\Psi _{s}$ is $[0,2s+1]$. For a positive constant $\xi $, if $ \Psi _{s}$ has $\xi s$ order of continuous derivatives, then for any $0\leq \text{\textrtaill} \leq s,~s\in \mathbb{N},$ we have
	
	\begin{equation}
		\int_{\mathbb{R}}x^{\text{\textrtaill} }\Psi _{s}(x)\mathrm{d}x=0.
		\label{h-1}
	\end{equation}%
	Evidently, when $s=0$, the system is reduced to the Haar system. Here, with regard to any $\Psi \in L_{\infty }(\mathbb{R}),$ we now have the conditions given by: (i) a finite positive $\xi $ having the property $\sup \Psi \subset \lbrack 0,\xi ],$ while (ii) its first $s$ moment vanishes. Furthermore, for $1\leq \text{\textrtaill} \leq s,~s\in \mathbb{N},$ we have $\int_{\mathbb{R}}t^{\text{\textrtaill} }\Psi (\text{\textrtailt})\mathrm{d}\text{\textrtailt}=0$ and $\int_{\mathbb{R}}\Psi (\text{\textrtailt})\mathrm{d}\text{\textrtailt}=1$. Therefore, by employing the Haar basis, the Baskakov type operators may be expressed by \cite{agratini}:
	
	\begin{equation}
		\left( \mathcal{L}_{m,\text{\textrtaill} }~\text{\textg}\right) (x)=m\sum_{%
			\text{\textrtaill} =0}^{\infty }\binom{m+\text{\textrtaill} -1}{\text{%
				\textrtaill} }\frac{x^{\text{\textrtaill} }}{(1+x)^{m+\text{\textrtaill} }}%
		\int_{\mathbb{R}}\text{\textg}\left( \text{\textrtailt}\right) \Psi \left( m%
		\text{\textrtailt}-\text{\textrtaill} \right) \mathrm{d}\text{\textrtailt},
		\label{mn-1}
	\end{equation}%
	%
	%
	%
	%
	%
	in which the operators $\mathcal{L}_{m,\text{\textrtaill} }$ refer to the extensions with respect to Baskakov-Kantorovich operators. By considering the $\sup \Psi \subset \lbrack 0,\xi ],$ the operators $\mathcal{L}_{m,\text{\textrtaill} }$ may be expressed as written below \cite{agratini}:
	
	\begin{equation}  \label{mn-2}
		\left(\mathcal{L}_{m,\text{\textrtaill}} ~ \text{\textg}\right)(x)= \sum_{%
			\text{\textrtaill}=0}^{\infty}\binom{m+\text{\textrtaill}-1}{\text{%
				\textrtaill}} \frac{x^\text{\textrtaill}}{(1+x)^{m+\text{\textrtaill}}}
		\int_{0}^{\xi}\text{\textg}\left(\frac{\text{\textrtailt}+\text{\textrtaill}%
		}{m}\right)\Psi(\text{\textrtailt})\mathrm{d}\text{\textrtailt}.
	\end{equation}
	
	
	The current section mentions the $\text{\texthtq}$-Baskakov type operators by employing compactly-supported wavelets of Daubechies constructed in \cite{nasir}. Let $\int_{ \mathbb{R}}x^{s}\Psi _{k}(x)\mathrm{d}_{\text{%
			\texthtq}}x=0$ when $0\leq s \leq k$ for $k\in \mathbb{N}$ as well as $\text{\texthtq%
	}>0.$\newline
	
	With regard to $\Psi \in L_{\infty }(\mathbb{R}),$ we assume the conditions given below in terms of wavelets: (i) a finite positive $\xi $ having the property $\sup \Psi \subset \lbrack 0,\xi ];$ and (ii) its first $k$ moment vanishes. With respect to $1\leq s \leq k$ and $k\in \mathbb{N}$, we now obtain $\int_{\mathbb{R}}\text{\textrtailt}^{s}\Psi (\text{\textrtailt})\mathrm{d}_{\text{\texthtq}}\text{\textrtailt}=0$ as well as $\int_{\mathbb{R}}\Psi (\text{\textrtailt})\mathrm{d}_{\text{\texthtq}}\text{\textrtailt}=1$. Therefore, for all $1\leq s \leq k,~ k\in \mathbb{N}$ as well as $0<\text{\texthtq}<1,$ we construct the $\text{\texthtq}-$analogue of Baskakov-Kantorovich type wavelets operators given by:
	
	\begin{equation}  \label{2.9}
		\left(\mathcal{S}_{\text{\textlonglegr},s,\text{\texthtq}} ~ \text{\textg}%
		\right)(x)= [\text{\textlonglegr}]_\text{\texthtq} \sum_{s =0}^{\infty }
		\text{\texthtq}^{s-1}B_{\text{\textlonglegr},s,\text{\texthtq}}(x) \int_{%
			\mathbb{R}}\text{\textg}\left(\text{\textrtailt}\right)\Psi \left(\text{%
			\texthtq}^{s-1}[\text{\textlonglegr}]_\text{\texthtq}\text{\textrtailt}-[s]_%
		\text{\texthtq}\right)\mathrm{d}_\text{\texthtq}\text{\textrtailt}.
	\end{equation}
	
	Thus, our operators $\mathcal{S}_{\text{\textlonglegr},s ,\text{\texthtq}}(\text{\textg};x)$ extends the $\text{\texthtq}$-Baskakov-Kantorovich operators expressed in \eqref{operator-1}. With respect to the choices of $k=0$ as well as $\Psi $ Haar basis, we obtain the exact $\text{\texthtq}$ -Baskakov-Kantorovich operators $\mathcal{T}_{\text{\textlonglegr},s ,\text{\texthtq}}(\text{\textg};x)$ by \eqref{operator-1}. Additionally, for the choices $k=0,~ \text{\texthtq}=1$ as well as $\Psi $ Haar basis, we determined the Baskakov-Kantorovich operators $\mathcal{K}_{\text{\textlonglegr},s ,\text{\texthtq}}(\text{\textg};x)$ by \eqref{k1}. Considering the $\sup \Psi \subset \lbrack 0,\xi ],$ the operators $\mathcal{S}_{\text{\textlonglegr},s ,\text{\texthtq}}(\text{\textg};x)$ may be expressed as given below:
	
	\begin{equation}  \label{operator-3}
		\left(\mathcal{S}_{\text{\textlonglegr},s,\text{\texthtq}} ~ \text{\textg}%
		\right)(x)= \sum_{s =0}^{\infty } B_{\text{\textlonglegr},s,\text{\texthtq}%
		}(x) \int_{0}^{\xi } \text{\textg}\left(\frac{\text{\textrtailt}+[s]_\text{%
				\texthtq}}{\text{\texthtq}^{s-1}[\text{\textlonglegr}]_\text{\texthtq}}%
		\right)\Psi\left(\text{\textrtailt}\right) \mathrm{d}_\text{\texthtq}\text{%
			\textrtailt}.
	\end{equation}
	
	It is evident that by choosing $\text{\texthtq}=1$, we obtain classical Baskakov-Kantorovich wavelets operators $\mathcal{L}_{\text{\textlonglegr},s }$ by \eqref{mn-1} as well as \eqref{mn-2}.
	
	We need the following result of \cite{nasir}:
	
	\begin{theorem}
		\label{ttt1} Suppose $e_{j}=t^{j}$ when $0\leq j\leq k$ and $k\in \mathbb{N}$.
		Therefore, we obtain
		\begin{equation*}
			\left( \mathcal{S}_{r,s ,\text{\texthtq}}\;e_{j}\right) (x)=\left( \mathcal{V}_{r,s
				,\text{\texthtq}}\;e_{j}\right) (x),
		\end{equation*}%
		in which $x\in \lbrack 0,\infty )$ as well as the operators $\left( \mathcal{V}_{r,s
			,\text{\texthtq}}\;g\right) (x)$ defined as above.\label{Theorem 2.1}
	\end{theorem}
	
	\section{\textbf{Weighted $\text{\texthtq}$-Statistical approximation}}
	
	This section presents the statistical approximation of Wavelets Kantorovich $\text{\texthtq}$-Baskakov operators $\mathcal{S}_{\text{\textlonglegr},s,\text{\texthtq}}$ defined by (\ref{2.9}) employing a Bohman  Korovkin-type theorem proven in \cite{duman2006s}. \newline
	
	Suppose $N_{\text{\textg}}$ is the constant depending on the function $\text{\textg}$ and represent the weighted space of a real valued function $\text{\textg}$ expressed on $\mathbb{R}$ by $B_{\rho}(\mathbb{R})$ with the
	property that $|\text{\textg}(x)|\leq N_{\text{\textg}}\rho (x)$ for all $%
	x\in \mathbb{R}$.  Now, we take into consideration the weighted subspace $C_{\rho }(%
	\mathbb{R})$ of $B_{\rho }(\mathbb{R})$ which is provided as
	\begin{equation*}
		\begin{aligned} C_{\rho}(x)(\mathbb{R}) = \{\text{\textg} \in
			B_{\rho}(\mathbb{R}): \mathrm{\text{\textg}}~ \mathrm{continuous}~
			\mathrm{in}~ \mathbb{R}\}. \end{aligned}
	\end{equation*}%
	This fits with the norm $\parallel .\parallel $, where $\parallel \text{\textg}%
	\parallel _{\rho }=\sup\limits_{x\in \mathbb{R}}\dfrac{|\text{\textg}(x)|}{%
		\rho (x)}$ and both $C_{\rho }(\mathbb{R})$ and $B_{\rho }(\mathbb{R})$ are
	Banach spaces. By the use of A-statistical convergence, Duman and Orhan \cite{duman2006s} proved the theorem given below, which is useful in proving our main result.
	
	\begin{theorem}
		(Duman and Orhan \cite{duman2006s}).  If $A = (a_{j\text{\textlonglegr}})_{j,\text{\textlonglegr}}$ is a positive regular summability matrix, then let $(L_{\text{\textlonglegr}})_{\text{\textlonglegr}}$ denote a sequence of positive linear operators from $C_{\rho_{1}}(\mathbb{R})$ to $B_{\rho_{2}}(\mathbb{R})$, in which $\rho_{1}$ as well as $\rho_{2}$ satisfies $\lim\limits_{|x|\to \infty}\dfrac{\rho_{1}}{\rho_{2}}$ = 0. Then
		\begin{equation*}
			\begin{aligned} st_{A} - \lim\limits_{\text{\textlonglegr}}\parallel
				L_{\text{\textlonglegr}}\text{\texthtq} - \text{\texthtq}\parallel
				_{\rho_{2}} = 0, ~ \forall \text{\texthtq} \in C_{\rho_{1}}(\mathbb{R})
			\end{aligned}
		\end{equation*}
		if and only if
		\begin{equation*}
			\begin{aligned} st_{A} - \lim\limits_{\text{\textlonglegr}}\parallel
				L_{\text{\textlonglegr}}H_{v} - H_{v}\parallel _{\rho_{1}} = 0 ~ for ~
				v=0,1,2, \end{aligned}
		\end{equation*}
		in which $H_{v} = \dfrac{x^{v}\rho_{1}(x)}{1 + x^{2}}$.\label{Theorem 3.1}
	\end{theorem}
	
	By examining this result, it is clear that if $\mathbb{R}$ is substituted by $ \mathbb{R_{+}}$, then the theorem holds true. Also, by analyzing Lemma \ref {Lemma 2.1}, we see that the sequence of operators $(\mathcal{S}_{\text{\textlonglegr},s ,\text{\texthtq}})_{\text{\textlonglegr}}$ fail to satisfy the properties of Bohman-Korovkin theorem. Now, let us take into consideration the weight functions $\rho_{0}(x) = 1 + x^{2}$ and $ \rho_{\alpha}(x) = 1 + x^{2+\alpha}$ for $x\in \mathbb{R_{+}}$ and $\alpha>0$ together with the remark below.
	
	\begin{remark}
		It is true that for $\text{\texthtq}\in(0,1)$, then $\lim\limits_{\text{%
				\textlonglegr}\to \infty}[\text{\textlonglegr}]_\text{\texthtq}=0$ or $%
		\dfrac{1}{1-\text{\texthtq}}$. Now, we consider the sequence $ (\text{%
			\texthtq}_{\text{\textlonglegr}})_{\text{\textlonglegr}}$ for $\text{\texthtq%
		}_{\text{\textlonglegr}}\in(0,1)$ with the property that $st -
		\lim\limits_{\text{\textlonglegr}\to\infty}\text{\texthtq}_{\text{%
				\textlonglegr}}=1$ and $st -  \lim\limits_{\text{\textlonglegr}\to\infty}%
		\text{\texthtq}_{\text{\textlonglegr}}^{\text{\textlonglegr}}=1$. Based on
		these facts, we have $ \lim\limits_{\text{\textlonglegr}\to\infty}[\text{%
			\textlonglegr}]_\text{\texthtq}=\infty$. This will help to check the convergence with respect to the operators expressed by inequality (\ref{2.9}). Thus, we now obtain the theorem given below: \label{Rmk 3.1}
	\end{remark}
	
	
	\begin{theorem}
		Suppose that the sequence $(\text{\texthtq}_{\text{\textlonglegr}})_{\text{%
				\textlonglegr}}$ satisfy Remark \ref{Rmk 3.1} above and $\mathcal{S}_{\text{%
				\textlonglegr},s ,\text{\texthtq}}$ is a positive linear operator. Then, we
		have:
		\begin{equation*}
			\begin{aligned} St_{\text{\texthtq}}-\lim\limits_{\text{\textlonglegr}}\parallel
				(\mathcal{S}_{\text{\textlonglegr},s ,\text{\texthtq}} (\text{\textg}) -
				\text{\textg}\parallel _{\rho_{\alpha}} = 0, ~ \forall \text{\textg} \in
				C_{\rho_{0}}(\mathbb{R_{+}}). \end{aligned}
		\end{equation*}
	\end{theorem}
	
	\begin{proof}
		Based on Lemma \ref{Lemma 2.1}(i) and Theorem \ref{Theorem 2.1}, we now have:
		\begin{equation*}
			\begin{aligned}
				\parallel (\mathcal{S}_{\text{\textlonglegr},s,\text{\texthtq}} (\text{\textg}) - \text{\textg}\parallel _{\rho_{0}} &=\sup\limits_{x \in \mathbb{R}} \dfrac{| (\mathcal{S}_{\text{\textlonglegr},s,\text{\texthtq}_{\text{\textlonglegr}}} e_{0})(x) - e_{0}(x)|}{1+x^{2}},\\
				& =\sup\limits_{x \in \mathbb{R}}  \dfrac{|1 - 1|}{1+x^{2}}, \\
				& = 0.
			\end{aligned}
		\end{equation*}
		In other words,
		\begin{equation*}
			St_{\text{\texthtq}}-\lim\limits_{\text{\textlonglegr}}\parallel (\mathcal{S}_{\text{%
					\textlonglegr},s,\text{\texthtq}} (\text{\textg}) - \text{\textg}\parallel
			_{\rho_{0}} =0.
		\end{equation*}%
		\newline
		Again, based on Lemma \ref{Lemma 2.1}(ii) and Theorem \ref{Theorem 2.1}, we now
		have:
		\begin{equation*}
			\begin{aligned}
				\parallel (\mathcal{S}_{\text{\textlonglegr},s,\text{\texthtq}} (\text{\textg}) - \text{\textg}\parallel _{\rho_{0}} &= \sup\limits_{x \in \mathbb{R}} \dfrac{| (\mathcal{S}_{\text{\textlonglegr},s,\text{\texthtq}_{\text{\textlonglegr}}} e_{1})(x) - e_{1}(x)|}{1+x^{2}},\\
				& =\sup\limits_{x \in \mathbb{R}}  \dfrac{|x - x|}{1+x^{2}}, \\
				& = 0.
			\end{aligned}
		\end{equation*}
		
		Using Lemma \ref{Lemma 2.1}(iii) and Theorem \ref{Theorem 2.1}, we now have:
		\begin{equation*}
			\begin{aligned}
				\parallel (\mathcal{S}_{\text{\textlonglegr},s,\text{\texthtq}} (\text{\textg}) - \text{\textg}\parallel_{\rho_{0}} &= \sup\limits_{x \in \mathbb{R}} \dfrac{\left| (\mathcal{S}_{\text{\textlonglegr},s,\text{\texthtq}_{\text{\textlonglegr}}} e_{2})(x) - e_{2}(x)\right|}{1+x^{2}},\\
				& =\sup\limits_{x \in \mathbb{R}}  \dfrac{\left| \left(  x^{2} + x\dfrac{1}{\left[ \text{\textlonglegr}\right] _{\text{\texthtq}_{\text{\textlonglegr}}}}\left( 1 + \dfrac{1}{\text{\texthtq}_{\text{\textlonglegr}}}x\right)\right)  - x^{2}\right|}{1+x^{2}}, \\
				& =\sup\limits_{x \in \mathbb{R}} \dfrac{\left|  \left( 1 +\dfrac{1}{\text{\texthtq}_{\text{\textlonglegr}}\left[ \text{\textlonglegr}\right] _{\text{\texthtq}_{n}}} - 1\right)  x^{2} + x\dfrac{1}{\left[ \text{\textlonglegr}\right] _{\text{\texthtq}_{\text{\textlonglegr}}}} \right|}{1+x^{2}}, \\
				& \leq \sup\limits_{x \in \mathbb{R}} \left|  \dfrac{1}{\text{\texthtq}_{\text{\textlonglegr}}\left[ \text{\textlonglegr}\right] _{\text{\texthtq}_{\text{\textlonglegr}}}}  x^{2} + x\dfrac{1}{\left[ \text{\textlonglegr}\right] _{\text{\texthtq}_{\text{\textlonglegr}}}} \right|,\\
				& \leq \sup\limits_{x \in \mathbb{R}} \left( \left|x^{2}\right|  \dfrac{1}{\text{\texthtq}_{\text{\textlonglegr}}\left[ \text{\textlonglegr}\right] _{\text{\texthtq}_{\text{\textlonglegr}}}}   + \left| x\right|\dfrac{1}{\left[ \text{\textlonglegr}\right] _{\text{\texthtq}_{\text{\textlonglegr}}}}\right),\\
				& = \left(\parallel e_{2}\parallel _{\rho_{0}} \dfrac{1}{\text{\texthtq}_{\text{\textlonglegr}}\left[ \text{\textlonglegr}\right] _{\text{\texthtq}_{\text{\textlonglegr}}}}   + \parallel e_{1}\parallel _{\rho_{0}}\dfrac{1}{\left[ \text{\textlonglegr}\right] _{\text{\texthtq}_{\text{\textlonglegr}}}}\right),\\
				& \leq \left( \dfrac{1}{\text{\texthtq}_{\text{\textlonglegr}}\left[ \text{\textlonglegr}\right] _{\text{\texthtq}_{\text{\textlonglegr}}}}  + \dfrac{1}{\left[ \text{\textlonglegr}\right] _{\text{\texthtq}_{\text{\textlonglegr}}}}\right).
			\end{aligned}
		\end{equation*}
		
		From Remark \ref{Rmk 3.1}, we have $st - \lim\limits_{\text{\textlonglegr}%
			\to\infty}\text{\texthtq}_{\text{\textlonglegr}}=1$. Furthermore, we also obtained $%
		\lim\limits_{\text{\textlonglegr}\to\infty}[\text{\textlonglegr}]_{\text{%
				\texthtq}}=\infty$. Consequently
		\begin{equation*}
			\begin{aligned}
				St_{\text{\texthtq}}-\lim\limits_{\text{\textlonglegr}}\parallel (\mathcal{S}_{\text{\textlonglegr},s,\text{\texthtq}} (\text{\textg}) - \text{\textg}\parallel_{\rho_{0}}  = 0.
			\end{aligned}
		\end{equation*}
		By employing Lemma \ref{Lemma 2.1} and also selecting $A = C_{1}$, known as the Ces\'{a}ro matrix of order one, $\rho_{0}(x) = 1 + x^{2}$, $\rho_{\alpha}(x) = 1 + x^{2+\alpha}$ for $x\in \mathbb{R_{+}}$ and $\alpha>0$, the proof may be immediately seen from Theorem \ref{Theorem 3.1}.
	\end{proof}
	
	\section{\protect\bigskip \textbf{The Convergence Rate}}
	
	In this sub-section, by means of weighted modulus of smoothness correlated to the space $B_{\rho\alpha}(\mathbb{R_{+}})$ and Lipschitz type maximal functions, we present the rates of statistical convergence with respect to the operators $\mathcal{S}_{\text{\textlonglegr},s,\text{\texthtq}}$ expressed by inequality (\ref{2.9}). The weighted modulus with respect to smoothness $\omega_{\rho_{\alpha}}$ correlated to the space $B_{\rho\alpha}(\mathbb{R_{+}})$ of a function $\text{\textg}$ is defined as:
	\begin{equation*} \label{3.1}
		\omega _{\rho _{\alpha }}(\text{\textg};\delta )=\sup\limits_{x\geq
			0,~0<i<\delta }\frac{\left\vert \text{\textg}(x+i)-\text{\textg}%
			(x)\right\vert }{1+(x+i)^{2+\alpha }},~\delta >0,\alpha \geq 0.
	\end{equation*}%
	It satisfies the following three axioms.
	
	\begin{enumerate}
		\item[(a)] $\omega_{\rho_{\alpha}}(\text{\textg}; \beta\delta) \leq (\beta +
		1)\omega_{\rho_{\alpha}}(\text{\textg}; \delta)$ for $\delta>0$ and $\beta>0$%
		.
		
		\item[(b)] $\omega_{\rho_{\alpha}}(\text{\textg}; \text{\textlonglegr}%
		\delta) \leq  \text{\textlonglegr}\omega_{\rho_{\alpha}}(\text{\textg};
		\delta)$ for $\delta>0$ and $\text{\textlonglegr} \in \mathbb{N}$.
		
		\item[(c)] $\lim\limits_{\delta\to\infty}\omega_{\rho_{\alpha}}(\text{\textg}%
		; \delta) =  0$.
	\end{enumerate}
	
	The following theorem gives an error estimate of an operator $\mathcal{S}_{\text{\textlonglegr},s,\text{\texthtq}}$ for the unbounded function $h$ by means of weighted modulus of smoothness correlated to the space $B_{\rho\alpha}(\mathbb{R_{+}})$.
	
	\begin{theorem}
		Suppose that $\text{\texthtq}\in(0,1)$ and $\alpha\geq0$. Then, for any $%
		\text{\textg}  \in B_{\rho\alpha}(\mathbb{R_{+}})$, we have
		\begin{equation*}
			\begin{aligned} \left|(\mathcal{S}_{\text{\textlonglegr},s,\text{\texthtq}}~
				\text{\textg}) (x) - \text{\textg}(x)\right| \leq
				\sqrt{\mathcal{S}_{\text{\textlonglegr},s,\text{\texthtq}}(\mu_{x,%
						\alpha}^{2};x)}\left( 1 + \dfrac{1}{\delta}
				\sqrt{\mathcal{S}_{\text{\textlonglegr},s,\text{\texthtq}}(\phi_{x}^{2};x)}
				\right) \omega_{\rho_{\alpha}}(\text{\textg}; \delta), \end{aligned}
		\end{equation*}
		where $\mu_{x,\alpha} (y) = 1 + \biggl(x + \left| y - x\right| \biggr)
		^{2+\alpha}$ as well as $\phi_{x}(y) = \left| y - x\right| $ for $y \geq 0$. \label%
		{Theorem 3.3}
	\end{theorem}
	
	\begin{proof}
		Suppose that $\text{\textlonglegr} \in \mathbb{N}$ and $\text{\textg} \in
		B_{\rho\alpha}(\mathbb{R_{+}})$. Using inequality (\ref{3.1}) and axiom (a)
		above, we can write that
		\begin{equation*}
			\begin{aligned}
				\left|\text{\textg}(y) - \text{\textg}(x)\right| & \leq \bigg( 1 + (x + \left| y - x\right| )^{2+\alpha}\bigg) \bigg(1 + \dfrac{1}{\delta}\left| y - x\right|\bigg) \omega_{\rho_{\alpha}}(\text{\textg}; \delta), \\
				& = \mu_{x,\alpha} (y)\bigg(1 + \dfrac{1}{\delta}\phi_{x}(y)\bigg)\omega_{\rho_{\alpha}}(\text{\textg}; \delta).
			\end{aligned}
		\end{equation*}
		Next, using the Cauchy inequality with respect to the positive linear operator's yields
		\begin{equation*}
			\begin{aligned}
				\left|(\mathcal{S}_{\text{\textlonglegr},s,\text{\texthtq}} ~\text{\textg}) (x) - \text{\textg}(x)\right| & \leq \left[ \text{\textlonglegr}\right]_{\text{\texthtq}} \sum_{s=0}^{\infty}\text{\texthtq}^{s-1} \upsilon_{s,\text{\textlonglegr}}^{\text{\texthtq}}(x)	\int_{\mathbb{R}}\left| \text{\textg}(y) - \text{\textg}(x)\right| \Psi\left( \left[ \text{\textlonglegr}\right] _{\text{\texthtq}}\frac{\text{\texthtq}^{s-1}}{1}y -\left[ s\right] _{\text{\texthtq}}\right) d_{\text{\texthtq}}y,\\
				& \leq \biggl(\mathcal{S}_{\text{\textlonglegr},s,\text{\texthtq}}(\mu_{x,\alpha};x) + \dfrac{1}{\delta}\mathcal{S}_{\text{\textlonglegr},s,\text{\texthtq}}(\mu_{x,\alpha}\phi_{x};x)\biggr)\omega_{\rho_{\alpha}}(\text{\textg}; \delta),\\
				& \leq \sqrt{\mathcal{S}_{\text{\textlonglegr},s,\text{\texthtq}}(\mu_{x,\alpha}^{2};x)}\left( 1 + \dfrac{1}{\delta} \sqrt{\mathcal{S}_{\text{\textlonglegr},s,\text{\texthtq}}(\phi_{x}^{2};x)} \right) \omega_{\rho_{\alpha}}(\text{\textg}; \delta).
			\end{aligned}
		\end{equation*}
	\end{proof}
	
	Now, we introduce the lemma given below, which may facilitate in proving the primary findings with respect to this research, since it is one of the facts which ensure that $(%
	\mathcal{S}_{\text{\textlonglegr},s,\text{\texthtq}}\text{\textg})(x) \in
	B_{\rho\alpha}(\mathbb{R_{+}})$ .
	
	\begin{lemma}
		Suppose that $0 < \text{\texthtq}  \leq 1$, then for $i, \text{\textlonglegr} \in \mathbb{N}$ and $x \in \mathbb{R_{+}}$, we obtain
		\begin{equation*}
			\begin{aligned} (\Gamma_{\text{\textlonglegr},s,\text{\texthtq}}e_{i})(x)
				\leq \dfrac{1}{[\text{\textlonglegr}]_{\text{\texthtq}}^{i-1}(1 +
					x)_{\text{\texthtq}}^{\text{\textlonglegr}}}x +
				\dfrac{2^{i-1}}{\text{\texthtq}^{i-1}}x(\Gamma_{\text{\textlonglegr}+1,s,%
					\text{\texthtq}}e_{i-1})(x). \label{Lemma 3.1} \end{aligned}
		\end{equation*}
	\end{lemma}
	
	\begin{proof}
		For $s \in \mathbb{N}$ as well as  $0 < \text{\texthtq} \leq 1$, we have the inequality given below:
		\begin{align}
			1\leq [s + 1]_{\text{\texthtq}} \leq 2[s]_{\text{\texthtq}}.  \label{3.2}
		\end{align}
		Now, let $i\in \mathbb{N}$. Using Equation (\ref{1.1}), we have:
		\begin{equation*}
			\begin{aligned}
				(\Gamma_{\text{\textlonglegr},s,\text{\texthtq}}e_{i})(x)  & = \sum_{s=0}^{\infty}\upsilon_{\text{\textlonglegr},s}^{\text{\texthtq}}(x)e_{i}\left(\frac{\left[ s\right] _{\text{\texthtq}}}{\text{\texthtq}^{s-1}\left[ \text{\textlonglegr}\right] _{\text{\texthtq}}}\right),  \\
				& = \sum_{s=0}^{\infty}\upsilon_{\text{\textlonglegr},s}^{\text{\texthtq}}(x)\left(\frac{\left[ s\right] _{\text{\texthtq}}}{\text{\texthtq}^{s-1}\left[ \text{\textlonglegr}\right] _{\text{\texthtq}}}\right)^{i},  \\
				& = \sum_{s=0}^{\infty}\upsilon_{\text{\textlonglegr},s}^{\text{\texthtq}}(x)\frac{\left[ s\right] _{\text{\texthtq}}^{i}}{\text{\texthtq}^{(s-1)i}\left[ \text{\textlonglegr}\right] _{\text{\texthtq}}^{i}},\\
				& = \sum_{s=1}^{\infty}x\upsilon_{\text{\textlonglegr}+1,s-1}^{\text{\texthtq}}(x)\frac{\left[ s\right] _{\text{\texthtq}}^{i-1}}{\text{\texthtq}^{(s-1)(i-1)}\left[ \text{\textlonglegr}\right] _{\text{\texthtq}}^{i-1}},\\
				& = \sum_{s=0}^{\infty}x\upsilon_{\text{\textlonglegr}+1,b}^{\text{\texthtq}}(x)\frac{\left[ s + 1\right] _{\text{\texthtq}}^{i-1}}{\text{\texthtq}^{s(i-1)}\left[ \text{\textlonglegr}\right] _{\text{\texthtq}}^{i-1}},\\
				& = \dfrac{1}{[\text{\textlonglegr}]_{\text{\texthtq}}^{i-1}(1 + x)_{\text{\texthtq}}^{\text{\textlonglegr}}}x + x\sum_{s=1}^{\infty}\upsilon_{\text{\textlonglegr}+1,s}^{\text{\texthtq}}(x)\frac{\left[ s + 1\right] _{\text{\texthtq}}^{i-1}}{\text{\texthtq}^{s(i-1)}\left[ \text{\textlonglegr}\right] _{\text{\texthtq}}^{i-1}}.
			\end{aligned}
		\end{equation*}
		Using Inequality (\ref{3.2}), we have,
		\begin{equation*}
			\begin{aligned}
				(\Gamma_{\text{\textlonglegr},s,\text{\texthtq}}e_{i})(x) & \leq \dfrac{x}{[\text{\textlonglegr}]_{\text{\texthtq}}^{i-1}(1 + x)_{\text{\texthtq}}^{\text{\textlonglegr}}} + x\sum_{s=1}^{\infty}\upsilon_{\text{\textlonglegr}+1,s}^{\text{\texthtq}}(x)\frac{(2\left[ s\right] _{\text{\texthtq}})^{i-1}}{\text{\texthtq}^{s(i-1)}\left[ \text{\textlonglegr}\right] _{\text{\texthtq}}^{i-1}},\\
				& = \dfrac{x}{[\text{\textlonglegr}]_{\text{\texthtq}}^{i-1}(1 + x)_{\text{\texthtq}}^{\text{\textlonglegr}}} + \dfrac{2^{i-1}}{\text{\texthtq}^{i-1}}x\sum_{s=1}^{\infty}\upsilon_{\text{\textlonglegr}+1,s}^{\text{\texthtq}}(x)\frac{\left[ s\right] _{\text{\texthtq}}^{i-1}}{\text{\texthtq}^{(s-1)(i-1)}\left[ \text{\textlonglegr}\right] _{\text{\texthtq}}^{i-1}}.
			\end{aligned}
		\end{equation*}
		Based on Equation (\ref{1.1}), we have that:
		\begin{equation*}
			\begin{aligned}
				(\Gamma_{\text{\textlonglegr}+1,s,\text{\texthtq}}e_{i-1})(x) =  \sum_{s=1}^{\infty}\upsilon_{\text{\textlonglegr}+1,s}^{\text{\texthtq}}(x)\frac{\left[ s\right] _{\text{\texthtq}}^{i-1}}{\text{\texthtq}^{(s-1)(i-1)}\left[ \text{\textlonglegr}\right] _{\text{\texthtq}}^{i-1}}.
			\end{aligned}
		\end{equation*}
		Consequently,
		\begin{equation*}
			\begin{aligned}
				(\Gamma_{\text{\textlonglegr},s,\text{\texthtq}}e_{i})(x) \leq \dfrac{1}{[\text{\textlonglegr}]_{\text{\texthtq}}^{i-1}(1 + x)_{\text{\texthtq}}^{\text{\textlonglegr}}}x + \dfrac{2^{i-1}}{\text{\texthtq}^{i-1}}x(\Gamma_{\text{\textlonglegr}+1,s,\text{\texthtq}}e_{i-1})(x).
			\end{aligned}
		\end{equation*}
	\end{proof}
	
	\begin{remark}
		Any positive and linear operator is monotone. Theorem \ref{Theorem 2.1} and
		Lemma \ref{Lemma 3.1} ensure that $(\mathcal{S}_{\text{\textlonglegr},s,%
			\text{\texthtq}}\text{\textg})(x) \in  B_{\rho\alpha}(\mathbb{R_{+}})$ for
		any $\text{\textg} \in B_{\rho\alpha}(\mathbb{R_{+}})  $ and $\alpha \in
		\mathbb{N}_{0}$, where $\mathbb{N}_{0} = \{0\} \cup  \mathbb{N}$. \label{Rmk
			3.2}
	\end{remark}
	
	We may state the major outcome of this portion as follows based on the preparation mentioned earlier:
	
	\begin{theorem}
		Let $(\text{\texthtq}_{\text{\textlonglegr}})_{\text{\textlonglegr}}$ be the
		sequence satisfying Remark \ref{Rmk 3.1} above and $\alpha \in \mathbb{N}_{0}
		$. Then, for every $\text{\textg} \in  B_{\rho\alpha}(\mathbb{R_{+}})$, we
		have
		\begin{equation*}
			\begin{aligned} \lim\limits_{\text{\textlonglegr}}\parallel
				(\mathcal{S}_{\text{\textlonglegr},s,\text{\texthtq}_{\text{\textlonglegr}}}
				\text{\textg})(x) - \text{\textg}(x)\parallel_{\rho_{\alpha}} \leq
				3C_{\alpha}\omega_{\rho_{\alpha}}(\text{\textg};
				\delta_{\text{\textlonglegr}}), \end{aligned}
		\end{equation*}
		where $C_{\alpha} > 0$ is a constant and $\delta_{\text{\textlonglegr}} =
		\sqrt{\dfrac{1}{ \text{\texthtq}_{\text{\textlonglegr}}[\text{\textlonglegr}%
				]_{\text{\texthtq}_{\text{\textlonglegr}}}}}$ . \label{Theorem 3.4}
	\end{theorem}
	
	\begin{proof}
		From Lemma \ref{Lemma 2.1}, we have the following:
		\begin{equation*}
			\begin{aligned}
				\mathcal{S}_{\text{\textlonglegr},s,\text{\texthtq}_{\text{\textlonglegr}}}(\phi_{x}^{2};x) & =  \left(  x^{2} + x\dfrac{1}{\left[ \text{\textlonglegr}\right] _{\text{\texthtq}_{\text{\textlonglegr}}}}\left( 1 + \dfrac{1}{\text{\texthtq}_{\text{\textlonglegr}}}x\right)\right)  - x^{2},\\
				& =  \left( 1 +\dfrac{1}{\text{\texthtq}_{\text{\textlonglegr}}\left[ \text{\textlonglegr}\right] _{\text{\texthtq}_{\text{\textlonglegr}}}} - 1\right)  x^{2} + x\dfrac{1}{\left[ \text{\textlonglegr}\right] _{\text{\texthtq}_{\text{\textlonglegr}}}},\\
				& = \dfrac{1}{\text{\texthtq}_{\text{\textlonglegr}}\left[ \text{\textlonglegr}\right] _{\text{\texthtq}_{\text{\textlonglegr}}}}  x^{2} + \dfrac{1}{\left[ \text{\textlonglegr}\right] _{\text{\texthtq}_{\text{\textlonglegr}}}}x.
			\end{aligned}
		\end{equation*}
		Consequently, we have the inequality:
		\begin{align}
			\mathcal{S}_{\text{\textlonglegr},s,\text{\texthtq}_\text{\textlonglegr}%
			}(\phi_{x}^{2};x) \leq \dfrac{1}{\text{\texthtq}_{\text{\textlonglegr}}\left[
				\text{\textlonglegr}\right] _{\text{\texthtq}_{\text{\textlonglegr}}}} x^{2}
			+ \dfrac{3}{\left[ \text{\textlonglegr}\right] _{\text{\texthtq}_{\text{%
							\textlonglegr}}}}x.  \label{In 3.3}
		\end{align}
		Let $\alpha \geq 0$ a constant and $\text{\textg} \in B_{\rho\alpha}(\mathbb{%
			R_{+}})$. Using Theorem \ref{Theorem 3.3} as well as the inequality (\ref{In 3.3})
		above, we may express the following:
		\begin{equation*}
			\begin{aligned}
				\lim\limits_{\text{\textlonglegr}}\parallel (\mathcal{S}_{\text{\textlonglegr},s,\text{\texthtq}} \text{\textg})(x) - \text{\textg}(x)\parallel_{\rho_{\alpha}}  & = \dfrac{\left| (\mathcal{S}_{\text{\textlonglegr},s,\text{\texthtq}} \text{\textg})(x) - \text{\textg}(x)\right|}{1+x^{2 + \alpha}},\\
				&  \leq \sqrt{\dfrac{\mathcal{S}_{\text{\textlonglegr},s,\text{\texthtq}}(\mu_{x,\alpha}^{2};x)}{1+x^{2+\alpha}}}\left( 1 + \dfrac{1}{\delta} \sqrt{\dfrac{\mathcal{S}_{\text{\textlonglegr},s,\text{\texthtq}}(\phi_{x}^{2};x)}{1+x^{1+\alpha}}} ~~\right) \omega_{\rho_{\alpha}}(\text{\textg}; \delta),  \\
				& \leq \sqrt{\dfrac{\mathcal{S}_{\text{\textlonglegr},s,\text{\texthtq}}(\mu_{x,\alpha}^{2};x)}{1+x^{2+\alpha}}}\left( 1 + \dfrac{1}{\delta} \sqrt{\left| \dfrac{1}{\text{\texthtq}_{\text{\textlonglegr}}\left[ \text{\textlonglegr}\right] _{\text{\texthtq}_{\text{\textlonglegr}}}}  x^{2} + \dfrac{3}{\left[ \text{\textlonglegr}\right] _{\text{\texthtq}_{\text{\textlonglegr}}}}x\right| } ~~\right), \\
				& \times \omega_{\rho_{\alpha}}(\text{\textg}; \delta),\\
				& \leq \sqrt{\dfrac{\mathcal{S}_{\text{\textlonglegr},s,\text{\texthtq}}(\mu_{x,\alpha}^{2};x)}{1+x^{2+\alpha}}}\left( 1 + \dfrac{1}{\delta} \sqrt{ \dfrac{1}{\text{\texthtq}_{\text{\textlonglegr}}\left[ \text{\textlonglegr}\right] _{\text{\texthtq}_{\text{\textlonglegr}}}}  \parallel e_{2}\parallel _{\rho_{\alpha}} + \dfrac{3}{\left[ \text{\textlonglegr}\right] _{\text{\texthtq}_{\text{\textlonglegr}}}}\parallel e_{2}\parallel _{\rho_{\alpha}} } ~~\right)\\
				& \times \omega_{\rho_{\alpha}}(\text{\textg}; \delta).\\
			\end{aligned}
		\end{equation*}
		Furthermore,
		\begin{equation*}
			\begin{aligned}
				\lim\limits_{\text{\textlonglegr}}\parallel (\mathcal{S}_{\text{\textlonglegr},s,\text{\texthtq}_{\text{\textlonglegr}}} \text{\textg})(x) - \text{\textg}(x)\parallel_{\rho_{\alpha}} & \leq \sqrt{\dfrac{\mathcal{S}_{\text{\textlonglegr},s,\text{\texthtq}}(\mu_{x,\alpha}^{2};x)}{1+x^{2+\alpha}}}\left( 1 + \dfrac{2}{\delta} \sqrt{ \dfrac{1}{\text{\texthtq}_{\text{\textlonglegr}}\left[ \text{\textlonglegr}\right] _{\text{\texthtq}_{\text{\textlonglegr}}}}} ~~\right)\omega_{\rho_{\alpha}}(\text{\textg}; \delta),\\
				& \leq \parallel \mathcal{S}_{\text{\textlonglegr},s,\text{\texthtq}}(\mu_{x,\alpha}^{2};x)\parallel _{\delta\alpha} \left( 1 + \dfrac{2}{\delta} \sqrt{ \dfrac{1}{\text{\texthtq}_{\text{\textlonglegr}}\left[ \text{\textlonglegr}\right] _{\text{\texthtq}_{\text{\textlonglegr}}}}} ~~ \right)\omega_{\rho_{\alpha}}(\text{\textg}; \delta).
			\end{aligned}
		\end{equation*}
		Let $C_{\alpha}= \parallel \mathcal{S}_{\text{\textlonglegr},s,\text{\texthtq%
		}}(\mu_{x,\alpha}^{2};x)\parallel_{\delta\alpha}$ and choose $\delta = \sqrt{%
			\dfrac{1}{\text{\texthtq}_{\text{\textlonglegr}}\left[ \text{\textlonglegr} %
				\right] _{\text{\texthtq}_{\text{\textlonglegr}}}}}$, we have:
		\begin{equation*}
			\begin{aligned}
				\lim\limits_{\text{\textlonglegr}}\parallel (\mathcal{S}_{\text{\textlonglegr},s,\text{\texthtq}_{\text{\textlonglegr}}} \text{\textg})(x) - \text{\textg}(x)\parallel_{\rho_{\alpha}} \leq 3C_{\alpha}\omega_{\rho_{\alpha}}(\text{\textg}; \delta_{\text{\textlonglegr}}).
			\end{aligned}
		\end{equation*}
	\end{proof}
	
	\begin{remark}
		Since $(\text{\texthtq}_{\text{\textlonglegr}})_{\text{\textlonglegr}}$
		satisfy Remark \ref{Rmk 3.1}, then the  sequence $(\delta_{\text{%
				\textlonglegr}})_{\text{\textlonglegr}}$ is statistically null, that is $st
		-  \lim\limits_{\text{\textlonglegr}}\omega_{\rho_{\alpha}}(\text{\textg};
		\delta_{\text{\textlonglegr}}) = 0$. Therefore,  Theorem \ref{Theorem 3.4}
		above gives the statistical rate of convergence of  $\mathcal{S}_{\text{%
				\textlonglegr},s,\text{\texthtq}_{\text{\textlonglegr}}}(x)$ to \text{\textg}%
		.
	\end{remark}

	\section{Graphical analysis}
	Using computer software, we will demonstrate some numerical examples with illustrative graphics.
	\begin{example}\label{eg1}
		Let $\text{\textg}(x)=(x-\frac{1}{5})(x-\frac{4}{9})$, $\text{\texthtq}=0.95$ and $n\in\{10,30,80\}$.
		The convergence of the operator towards the function $\text{\textg}(x)$ is shown in Figure \ref{fig1}.

	\begin{figure}[H]
		
		\includegraphics{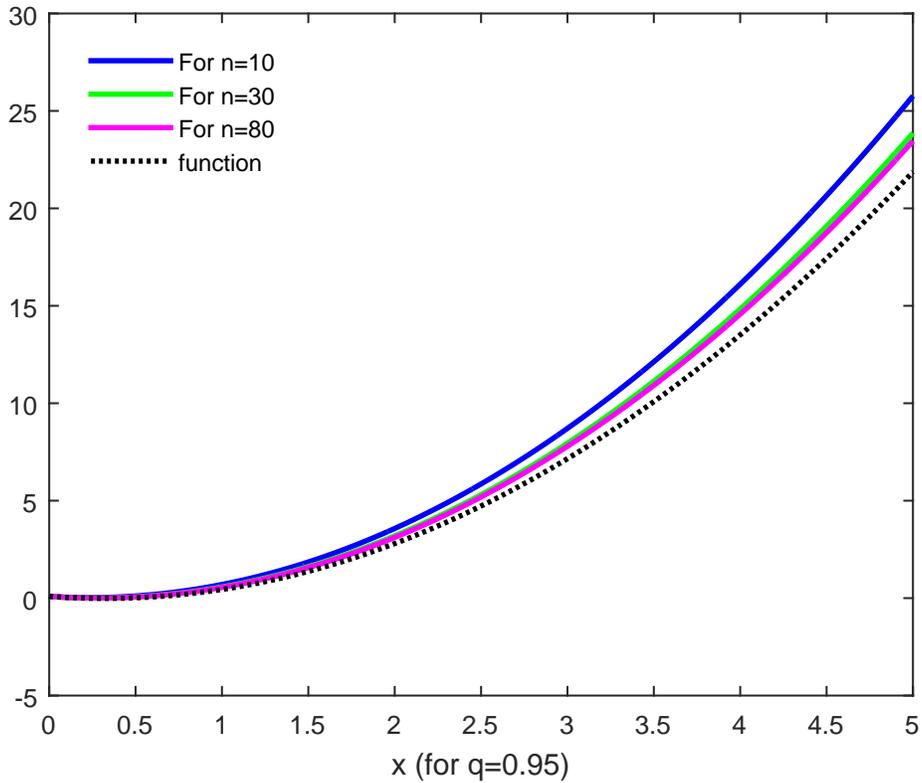}
		
		\caption{convergence of the operator towards the function $\text{\textg}(x)=(x-\frac{1}{5})(x-\frac{4}{9})$}
		\label{fig1}
	\end{figure}
\end{example}

	\begin{example}\label{eg2}
		Let $\text{\textg}(x)=x^2-1$, $\text{\texthtq}=1$ and $n\in\{10,30,60\}$.
		The convergence of the operator towards the function $\text{\textg}(x)$ is shown in Figure \ref{fig2}.

	\begin{figure}[H]
		
		\includegraphics{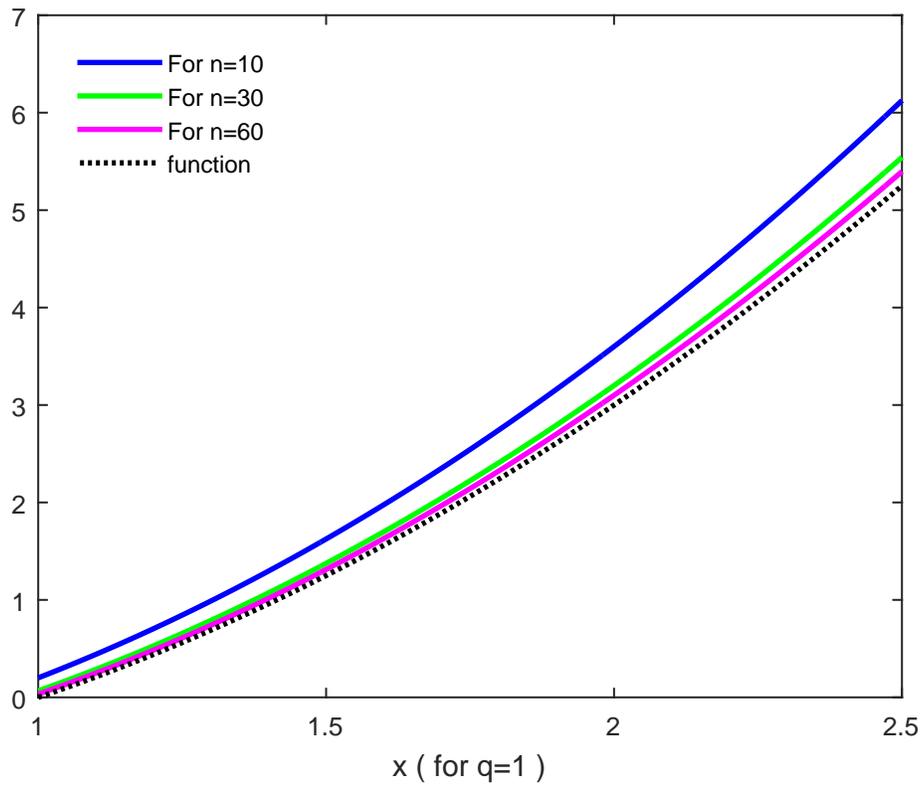}
		
		\caption{convergence of the operator towards the function $\text{\textg}(x)=x^2-1$}
		\label{fig2}
	\end{figure}
\end{example}
\newpage
\begin{example} \label{eg3}
		Let $f(x)=x^2-4x+3$. For $n=50$ and different values of \text{\texthtq},
		the convergence of the operator towards the function $f(x)$ is shown in Figure \ref{fig3}.

	\begin{figure}[H]
		
		\includegraphics{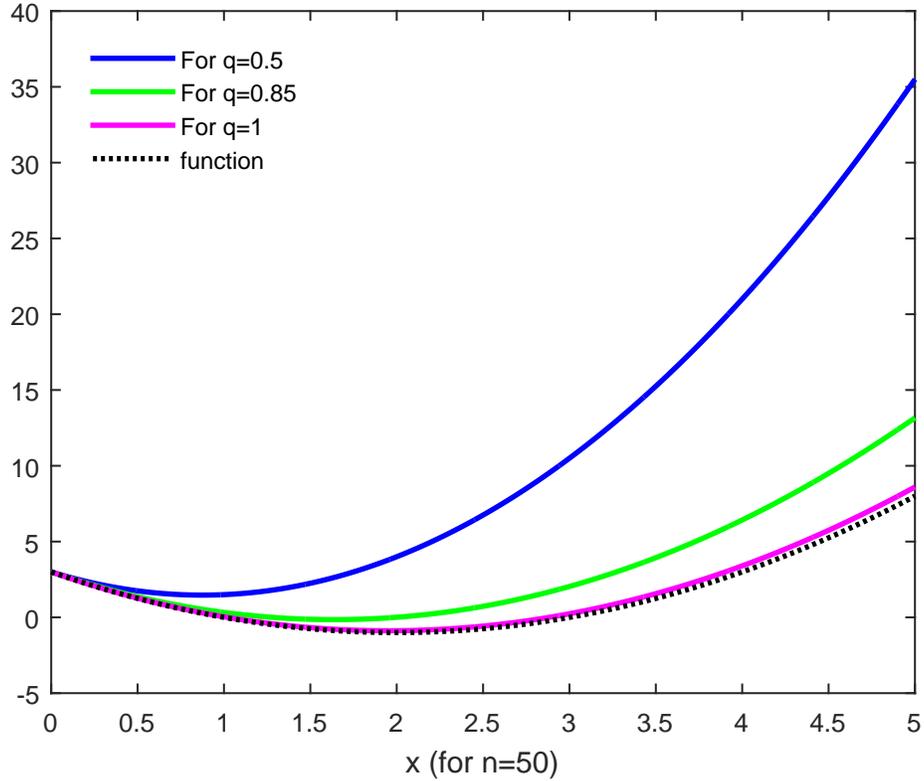}
		
		\caption{Convergence of the operator for different values of \text{\texthtq}}
		\label{fig3}
	\end{figure}
\end{example}

	\section{Conclusion}

With the facilitation of Bohman Korovkin-type theorem, the investigation on weighted statistical approximation behavior of Wavelets Kantorovich \text{\texthtq}-Baskakov operators $\mathcal{S}_{\text{\textlonglegr},s,\text{\texthtq}}$ is discussed under this study.
Moreover, the convergence statistical rate with respect to the operators $%
\mathcal{S}_{\text{\textlonglegr},s,\text{\texthtq}}$ is provided in this research with regard to the weighted modulus of smoothness correlated to the space $B_{\rho\alpha}(\mathbb{R_{+}})$. The statistical approximation properties discussed in this study are the same as those of classical \text{\texthtq}-Baskakov operators defined by equation (\ref{1.1}) since they share the same moments. \newline

	\section*{Declarations}
	\subsection*{Ethical Approval} Not Applicable
	\subsection*{Availability of supporting data} Not Applicable
	\subsection*{Competing interests} Not Applicable
	\subsection*{Funding} Not Applicable
	\subsection*{Acknowledgments}	Not Applicable

\end{document}